\documentclass[a4paper]{amsart}
  \usepackage{amssymb,exscale,bbm,mathrsfs,nicefrac,stmaryrd}
  \usepackage[symbol]{footmisc} 
  \usepackage[british]{babel}
  \usepackage{tikz}
     \usetikzlibrary{decorations.pathmorphing,patterns,mindmap,trees,arrows}

 \newtheorem{theorem}{Theorem}[section]
 
 \newtheorem{lemma}[theorem]{Lemma}
 \newtheorem{proposition}[theorem]{Proposition}
 \theoremstyle{definition}
 
 \theoremstyle{remark}

 \numberwithin{equation}{section}

\begin{document}

%
%
  \renewcommand {\Re}{\text{Re}}
  \renewcommand {\Im}{\text{Im}}
  \newcommand {\sfrac}[2] { {{}^{#1}\!\!/\!{}_{#2}}} 
  \newcommand {\pihalbe}{\sfrac{\pi}2}
  \newcommand {\einhalb}{\sfrac{1}2}
  \newcommand {\Borel}{{\mathcal Bo}}
  \newcommand {\Poisson}{{\mathscr P}}
  \newcommand {\CC}{\mathbb C}
  \newcommand {\DD}{\mathbb D}
  \newcommand {\EE}{\mathbb E}
  \newcommand {\NN}{\mathbb N}
  \newcommand {\ZZ}{\mathbb Z}
  \newcommand {\RR}{\mathbb R}
  \newcommand {\calB}{\mathcal B}
  \newcommand {\calC}{\mathcal C}
  \newcommand {\calE}{\mathcal E}
  \newcommand {\calT}{\mathcal T}
  \newcommand {\calS}{\mathscr S}
  \newcommand {\BOUNDED}{\mathcal B}
  \newcommand {\DOMAIN}{\mathcal D}
  \newcommand {\FOURIER}{\mathcal F}
  \newcommand {\LAPLACE}{\mathcal L}
  \newcommand {\eins} {\mathbbm 1}
  \newcommand {\al}{\alpha}
  \newcommand {\la}{\lambda}
  \newcommand {\eps}{\varepsilon}
  \newcommand {\calO} {\mathcal O}
  \newcommand {\Ga}{\Gamma}
  \newcommand {\Hinfty}{\mathrm{H}^\infty}
  \newcommand {\ga}{\gamma}
  \newcommand {\om}{\omega}
  \newcommand {\Xm}{X_{-1}}
  \newcommand {\eR}{\mathcal R}
  \newcommand {\Sec}[1] {S({#1})}
  \newcommand {\norm}[1] {\| #1 \|}  
  \newcommand {\lrnorm}[1]{\left\| #1 \right\|}
  \newcommand {\bignorm}[1]{\bigl\| #1 \bigr\|}
  \newcommand {\Bignorm}[1]{\Bigl\| #1 \Bigr\|}
  \newcommand {\Biggnorm}[1]{\Biggl\| #1 \Biggr\|}
  \newcommand {\biggnorm}[1]{\biggl\| #1 \biggr\|}
  \newcommand {\Bigidual}[3] {\Bigl\langle #1, #2 \Bigr\rangle_{#3}}
  \newcommand {\bigidual}[3] {\bigl\langle #1, #2 \bigr\rangle_{#3}}
  \newcommand {\idual}[3] {\langle #1, #2 \rangle_{#3}}
  \newcommand {\Bigdual}[2] {\Bigidual{#1}{#2}{}}
  \newcommand {\bigdual}[2] {\bigidual{#1}{#2}{}}
  \newcommand {\dual}[2] {\idual{#1}{#2}{} }
  \newcommand {\Id} {\mathrm{Id}}
  \newcommand {\suchthat}{:\;}
  \newcommand {\stolz}[1] {\mathrm{Stolz}_{#1}}
  \newcommand{\GB}[1]{\left\llbracket\,#1\,\right\rrbracket^{\gamma}} 
  \newcommand{\RB}[1]{\left\llbracket\,#1\,\right\rrbracket^{\eR}} 
  \newcommand{\ud}{\mathrm{d}}
  \newcommand{\ue}{\mathrm{e}}
  \newcommand{\ui}{\mathrm{i}}
  \newcommand{\sprod}[2]{\left[#1|#2\right]}
  \newcommand{\dprod}[2]{\left\langle#1,#2\right\rangle}

\renewcommand{\labelenumi} {(\alph{enumi})}    
\renewcommand{\labelenumii}{(\roman{enumii})}
\renewcommand{\theenumi} {(\alph{enumi})}      
\renewcommand{\theenumii}{(\roman{enumii})}    


%
%

\newcounter{aufzi}
\newenvironment{aufzi}{\begin{list}{ {\upshape(\alph{aufzi})}}{
        \usecounter{aufzi}
        \topsep1ex
        \parsep0cm
        \itemsep1ex
        \leftmargin0.8cm
        \labelwidth0.5cm
        \labelsep0.3cm
}}
{\end{list}}

\newcounter{aufzii}
\newenvironment{aufzii}{\begin{list}{\hfill {\upshape 
(\roman{aufzii})}}{
        \usecounter{aufzii}
        \topsep1ex
        \parsep0cm
        \itemsep1ex
        \leftmargin0.8cm
        \labelwidth0.5cm
        \labelsep0.3cm
         \itemindent0cm
}}
{\end{list}}

\renewcommand{\theaufzi}{(\alph{aufzi})}
\renewcommand{\theaufzii}{(\roman{aufzii})}

\newcommand\irregularcircle[2]{
  \pgfextra {\pgfmathsetmacro\len{(#1)+rand*(#2)}}
  +(0:\len pt)
  \foreach \a in {10,20,...,350}{
    \pgfextra {\pgfmathsetmacro\len{(#1)+rand*(#2)}}
    -- +(\a:\len pt)
  } -- cycle
}
\pgfmathsetmacro\sprayRadius{.2pt}
\pgfmathsetmacro\sprayPeriod{.5cm}
\pgfdeclarepatternformonly{spray}{\pgfpoint{-\sprayRadius}{-\sprayRadius}}{\pgfpoint{1cm + \sprayRadius}{1cm + \sprayRadius}}{\pgfpoint{\sprayPeriod}{\sprayPeriod}}{
    \foreach \x/\y in {2/53,6/52,11/48,23/49,20/47,32/46,41/47,47/51,56/52,46/44,4/43,16/42,33/41,41/37,49/35,55/31,37/35,44/30,28/37,24/36,17/37,7/38,0/31,8/29,18/31,28/30,37/28,30/27,46/24,51/21,24/23,12/24,4/21,18/19,12/16,31/21,38/18,26/16,46/16,56/12,52/10,45/8,51/4,37/12,35/7,24/9,14/9,2/12,8/6,15/4,27/0,34/1,40/1} {
        \pgfpathcircle{\pgfpoint{(\x + random()) / 57 * \sprayPeriod}{\sprayPeriod - (\y + random()) / 55 * \sprayPeriod}}{\sprayRadius}
    }
    \pgfusepath{fill}
}

\newcommand{\RittKalkuel}{
  \begin{tikzpicture}[scale=3]
      \draw (0,0) circle (1);
      \pgfmathsetlengthmacro{\A}{2}
      \draw[samples=500,scale=28.5,domain=0:720, smooth, variable=\t] plot (  {cos(\t)*(-sqrt(-2*\A*\A * cos(\t) + 2*\A*\A + cos(\t)*cos(\t) - 1) + \A*\A - cos(\t))/(\A*\A-1)} , {sin(\t)*(-sqrt(-2*\A*\A * cos(\t) + 2*\A*\A + cos(\t)*cos(\t) - 1) + \A*\A - cos(\t))/(\A*\A-1)} );

            \pgfmathsetlengthmacro{\A}{3}
      \draw[samples=500,scale=28.5,domain=0:200, smooth, variable=\t, thick, ->] plot (  {cos(\t)*(-sqrt(-2*\A*\A * cos(\t) + 2*\A*\A + cos(\t)*cos(\t) - 1) + \A*\A - cos(\t))/(\A*\A-1)} , {sin(\t)*(-sqrt(-2*\A*\A * cos(\t) + 2*\A*\A + cos(\t)*cos(\t) - 1) + \A*\A - cos(\t))/(\A*\A-1)} );
      \draw[samples=500,scale=28.5,domain=200:720, smooth, variable=\t, thick] plot (  {cos(\t)*(-sqrt(-2*\A*\A * cos(\t) + 2*\A*\A + cos(\t)*cos(\t) - 1) + \A*\A - cos(\t))/(\A*\A-1)} , {sin(\t)*(-sqrt(-2*\A*\A * cos(\t) + 2*\A*\A + cos(\t)*cos(\t) - 1) + \A*\A - cos(\t))/(\A*\A-1)} );

            \pgfmathsetlengthmacro{\A}{4}
      \draw[samples=500,scale=28.5,domain=0:720, smooth, variable=\t, dotted,pattern=spray] plot (  {cos(\t)*(-sqrt(-2*\A*\A * cos(\t) + 2*\A*\A + cos(\t)*cos(\t) - 1) + \A*\A - cos(\t))/(\A*\A-1)} , {sin(\t)*(-sqrt(-2*\A*\A * cos(\t) + 2*\A*\A + cos(\t)*cos(\t) - 1) + \A*\A - cos(\t))/(\A*\A-1)} );

    \draw[lightgray, fill=lightgray] plot [smooth cycle] coordinates {(0.99, 0) (0.5,0.3) (-0.2, 0.1) (-0.1, -0.38)  (0.44,-0.2)  } ;

      \draw[->] (-1.2, 0) -- (1.2, 0) node[right] {$x$};
      \draw[->] (0, -1.2) -- (0, 1.2) node[above] {$y$};
      \draw (0.3, 0.1) node {$\sigma(T)$} ;
      \draw (-0.55, 0.1) node {$\gamma$};
    \end{tikzpicture}
}

\newcommand{\StolzWinkel}{
%
    \begin{tikzpicture}[scale=3]
      \pgfmathsetlengthmacro{\A}{2}
      \draw[thick]  (0,0) circle (1);
      \draw[lightgray, fill=lightgray]   (1, 0) to (0.75,  0.433012701892219 ) to (0.75,  -0.433012701892219 ) --cycle ;   
      \draw[black, fill=lightgray]  (0, 0) circle (0.866025403784439) ; 
      \draw[fill=gray, thick, samples=500,scale=28.5,domain=0:720, smooth, variable=\t] plot (  {cos(\t)*(-sqrt(-2*\A*\A * cos(\t) + 2*\A*\A + cos(\t)*cos(\t) - 1) + \A*\A - cos(\t))/(\A*\A-1)} , {sin(\t)*(-sqrt(-2*\A*\A * cos(\t) + 2*\A*\A + cos(\t)*cos(\t) - 1) + \A*\A - cos(\t))/(\A*\A-1)} );
      \draw (0.2, -0.2) node {$\stolz{\alpha}$};
      \draw   (1,0) to (0.25, 1.299038105676658) ;
      \draw   (1,0) to (0.25,-1.299038105676658) ;
      \draw   (0.75,  0.433012701892219 ) to (0,0) ; 
      \draw   (0.375, 0.22) node[above] {$r$} ; 
      \begin{scope}[even odd rule]
        \clip  { (0,0) to (0.75,  0.433012701892219 ) to  (0.5, 0.866025403784439) --cycle };
        \draw (0.75,  0.433012701892219 ) circle (0.07)  ;
        \draw[fill=black] (0.715, 0.445) circle (0.005) ; 
      \end{scope}
      \draw   (-0.5, 0.5) node {$B_r$} ;
      \draw  (0, 0) circle (0.866025403784439) ; 
      \draw[->] (-1.2, 0) -- (1.2, 0) node[right] {$x$};
      \draw[->] (0, -1.2) -- (0, 1.2) node[above] {$y$};
  \end{tikzpicture}
}

\allowdisplaybreaks

\title[Ritt operators and their functional calculus]
      {Remarks on Ritt operators, their $\Hinfty$ functional calculus and associated square function  estimates}
\author[Bernhard H. Haak]{Bernhard H. Haak}

\address{%
Institut de Math\'ematiques de Bordeaux\\Universit\'e Bordeaux\\351, cours de la Lib\'eration\\33405 Talence CEDEX\\FRANCE}

\email{bernhard.haak@math.u-bordeaux.fr}



\date{\today}

\begin{abstract}
  This note deals with the boundedness of the $\Hinfty$ functional
  calculus of Ritt operators $T$ and associated square function
  estimates. The purpose is to give a shorter, concise and slightly
  more general approach towards Le~Merdy's results \cite{LeMerdy:Ritt2014}
  on this subject.
\end{abstract}

\maketitle

\setcounter{footnote}{1}

\section{Introduction and main result}

Ritt --- or Ritt-Tadmor --- operators and their functional calculus
have attracted some attention in recent years, see for example
\cite{Arhancet:Ritt,ArhancetLeMerdy:Ritt,ArrigoniLeMerdy:Ritt,AssaniHallyburtonMcMahonSchmidtSchoone,GomilkoTomilov,KaltonPortal:Ritt,LancienLeMerdy:Ritt,LeMerdy:Ritt2014,MohantyRay:Ritt,Schwenninger:Ritt,Vitse:Ritt}.
It is well known that the boundedness of the $\Hinfty$ functional
calculus of sectorial or strip-type operators is linked to certain
square function estimates, see for example
\cite{KaltonWeis:square-function-est,LeMerdy:square-functions,McIntosh:H-infty-calc},
as well as \cite{Haase:Buch,FabulousFour:Band2} for extensive
references. A 'discrete' analogue to such results for Ritt operators
is stated next.

\begin{theorem}[Le~Merdy]\label{thm:LeMerdy}
  Let $1<p< \infty$ and $T: L^p(\Omega)\to L^p(\Omega)$ be a Ritt
  operator. Then the following assertions are equivalent.
\begin{aufzi}
\item The operator $T$ admits a bounded
  $\Hinfty(B_\gamma)$-functional calculus for some
  $\gamma \in( 0, \pi/2)$ where $B_\gamma$ is a certain Stolz type
  domain within the complex unit circle (see figure~\ref{fig:2Stolz} below).
\item The operator $T$ as well as its adjoint
  $T^*: L^q(\Omega) \to L^q(\Omega)$ both satisfy uniform estimates
  \begin{align*}
               & \quad \lrnorm{ \Bigl(\sum_{k=1}^\infty k \bigl| (\Id-T)T^{k-1} f |^2 \Bigr)^\einhalb }_{L^p} \lesssim  \norm{f}_{L^p} \\
    \text{and} & \quad \lrnorm{ \Bigl(\sum_{k=1}^\infty k \bigl| (\Id-T^*)(T^*)^{k-1} g |^2 \Bigr)^\einhalb }_{L^q} \lesssim  \norm{g}_{L^q}
  \end{align*}
\end{aufzi}
\end{theorem}

We will formulate Le~Merdy's result for Banach spaces having certain
geometric conditions (see section~\ref{sec:geometry} for precise definitions).
The basic idea is to replace the square sum
\begin{equation}
  \label{eq:Lp-sqf}
   \lrnorm{ \Bigl(\sum_{k=1}^\infty  \bigl| f_k |^2 \Bigr)^\einhalb }_{L^p}
\end{equation}
by a random sum of the type
\[
  \EE \lrnorm{ \sum_{k=1}^\infty  r_k f_k  }_{X}
  \qquad\text{or}\qquad
  \EE \lrnorm{ \sum_{k=1}^\infty  \gamma_k f_k  }_{X}
\]
where the variables $r_k$ (respectively, $\gamma_k$) refer to a
sequence of independent, identically distributed Rademacher
(respectively standard Gaussian) random variables. In the case that
$X = L^p(\Omega)$, and, more generally, the case of Banach spaces of
finite cotype, these two random sums are comparable in the sense of a
double inequality. Moreover, both expressions are equivalent to
\eqref{eq:Lp-sqf} if $X=L^p(\Omega)$ with $1<p<\infty$.

\medskip

The proposed extension in this article can be split up in several
specific sub-statements under different hypotheses on the operators
and the Banach space geometry, and are the subject of  sections
\ref{sec:Hinfty-to-sqf}--\ref{sec:form-sqf-to-hinfty}.  A framework
that contains all of these results, and allows us to compare our
findings with Le~Merdy's Theorem~\ref{thm:LeMerdy} is the case of a
Banach space enjoying Pisier's property $(\alpha)$, as defined
below. Our main result can now be phrased as follows:

\begin{theorem}\label{thm:main}
  Let $X$ be a Banach space that has Pisier's property $(\alpha)$, and
  $T\in \BOUNDED(X)$ be a bounded operator on $X$, such that $(I-T)$
  is injective and has dense range. Then the following assertions are
  equivalent.
\begin{aufzi}
\item The operator $T$ admits a bounded
  $\Hinfty(\stolz{\omega})$-functional calculus for some
  $\omega \in( 0, \pi/2)$ where $\stolz{\omega}$ is a Stolz type domain, see
  figure~\ref{fig:2Stolz} below.
\item For some (and hence all) $m_1, m_2 \ge 1$, the operator $T$ and its adjoint
  $T'$  both satisfy  uniform estimates
  \begin{align*}
               & \quad \EE \Bignorm{ \sum \gamma_k  k^{m_1-\einhalb} (\Id-T)^{m_1} T^{k-1} x }_{X}^2    \le   C_{m_1} \norm{x}\\
    \text{and} & \quad \EE \Bignorm{ \sum \gamma_k  k^{m_2-\einhalb} (\Id-T')^{m_2} (T')^{k-1} x' }_{X'}^2 \le C_{m_2} \norm{x'}_{X'}.
  \end{align*}
\end{aufzi}
Moreover, $T$ is an $\eR$-Ritt and hence $\gamma$-Ritt operator in this case.
\end{theorem}

Notice that if $T$ is weakly compact, the mean ergodic theorem
\cite{Yosida:ergodic1938} shows that $X$ decomposes as
\[
  X = \ker(\Id-T) \oplus \overline{ \text{Im}(\Id-T) } =: X_0 \oplus X_1.
\]
One can then define $f(T) := f(1)\Id \oplus f(T|_{X_1})$ and observe
that $(\Id-T|_{X_1})$ is injective with dense range. This allows to
remove the hypothesis of injectivity and dense range of the operator
$\Id{-}T$ in Theorem~\ref{thm:main} at the expense of requiring that
$X$ is a reflexive Banach space.

\medskip

The fact that Theorem~\ref{thm:LeMerdy} can be formulated in a context
of Banach spaces with property $(\alpha)$ already appears in an
unnumbered remark in \cite[end of section~7]{LeMerdy:Ritt2014}. The
emphasis of our approach is therefore on different proofs that yield a
slightly more general result. As an example, in contrast with
Theorem~\ref{thm:LeMerdy}, in our Theorem~\ref{thm:main} the
$\eR$-Ritt property is no longer a hypothesis, but a conclusion.
Observe that $\eR$-boundedness or $\gamma$-boundedness, (and hence
the $R /\gamma$--Ritt property) is usually hard to check in concrete
examples: we consider our version an important improvement for this
reason.

Our approach also allows a bit more freedom in the choice of the
square function, and a sharper domain of the $\Hinfty$-functions (see
below for a discussion of different domain types in the literature).
Moreover, our proofs rely on the same structural arguments that we use
in \cite{HaakHaase} to explain the link between $\Hinfty$-calculus
and (dual) square function estimates for sectorial or strip-type
operators. Finally, our approach avoids the 'detour' to the functional
calculus of the sectorial operator $A = \Id{-}T$ that Le~Merdy
\cite{LeMerdy:Ritt2014} uses at some steps in his proofs. Instead, we
constantly stay in the ``Ritt world''.

\medskip

Acknowledgement: the author would like to thank Markus Haase for many
helpful discussions on functional calculi and square functions
estimates.

\section{Basic definitions from the geometry of Banach spaces}\label{sec:geometry}
We recall some standard terminology from geometry of Banach spaces, cf.
\cite[Chapter~11]{DiestelJarchowTonge} or \cite[Chapter~7]{FabulousFour:Band2}.

A Banach space $X$ is said to have Rademacher type $p \in [1,2]$,
if there exists a constant $t_p(X)$ such that, for
all $N\ge 1$ and all $x_1, \ldots, x_N \in X$
\[
  \EE \Bignorm{ \sum_{n=1}^N r_n x_n } \le t_p(X) \Bigl( \sum_{n=1}^N \norm{x_n}^p \Bigr)^{\nicefrac{1}{p} }
\]
All Banach spaces have type $p=1$. We say the the type is non-trivial
if $p>1$.  Similarly, $X$ has cotype $q \ge 2$ if there exists a
constant $c_q(X)$ such that, for all $N\ge 1$ and all
$x_1, \ldots, x_N \in X$
\[
  \Bigl( \sum_{n=1}^N \norm{x_n}^q \Bigr)^{\nicefrac{1}{q} }   \le c_q(X) \EE \Bignorm{ \sum_{n=1}^N r_n x_n }
\]
with the obvious modification if $q={+}\infty$. All Banach spaces have
cotype $q={+}\infty$. There are
two reasons to consider such spaces in the following.

First, since spaces like $c_0, \ell_\infty$ have trivial cotype, no
Banach space with finite cotype can contain an isomorphic copy of such
spaces. The Kwapién -- Hoffman$\,$Jørgensen theorem
\cite{Hoffmann-Jorgensen,Kwapien:c_0} ensures therefore that on Banach
spaces with non-trivial cotype, uniform boundedness (say, in
$L^2\Omega; X)$-norm), and convergence of sequence of random sums
\[
  S_N = \sum_{n=1}^N r_n x_n
\]
are equivalent. Second, on Banach spaces with finite cotype, Gaussian
and Rademacher random sums
\[
  \EE \Bignorm{ r_n x_n }_X \qquad\text{and}\qquad
  \EE \Bignorm{ \gamma_n x_n }_X
\]
are equivalent with constants depending only on the cotype constant,
see for example \cite[Theorem~8.1.3]{FabulousFour:Band2}.

\bigskip

We remind the reader of Kahane's contraction principle: for any Banach
space, any $N\ge 1$ and $x_1, \ldots x_N \in X$ and any sequence
$(\alpha_n)_{n\ge 1}$ in $\RR^\NN$,
\begin{equation}   \label{eq:contraction-principle}
  \EE \Bignorm{ \sum_{n=1}^N \alpha_n r_n x_n }
  \le
  \max\{ |\alpha_n|: 1 \le n \le N\} \; \EE \Bignorm{ \sum_{n=1}^N  r_n x_n }.
\end{equation}
If a similar ``contraction'' holds for a collection $\calT$ of
operators, instead of scalars, i.e. if $\calT \subset \BOUNDED(X; Y)$ is such that
any $N\ge 1$ and $x_1, \ldots x_N \in X$ and any  $T_1, \ldots, T_N \in \calT$
\[
  \EE \Bignorm{ \sum_{n=1}^N  r_n T_n x_n }_Y
  \le
  C \; \EE \Bignorm{ \sum_{n=1}^N  r_n x_n }_X.
\]
we say that $\calT$ is $\eR$-bounded. Letting $N{=}1$ shows that
$\eR$-bounded sets are bounded in $\BOUNDED(X)$, but the converse is
wrong, in general.  The corresponding expression with Gaussians
instead of Rademachers leads to the notion of
$\gamma$-boundedness. Rademacher- or $\eR$-boundedness implies
$\gamma$-boundedness, but in spaces of finite cotype, both notions
coincide. We need at some places that
\begin{equation}  \label{eq:convex-hull}
  \calT \quad \eR\text{--bounded} \qquad \Rightarrow \qquad  \overline{ \text{absco}(\calT) } \quad \eR\text{--bounded},
\end{equation}
a statement that is commonly referred to as ``convex hull lemma'', see for
example \cite{ClementDePagterSukochevWitvliet,FabulousFour:Band2}.

\bigskip

Finally, we recall the definition of Pisier's property $(\alpha)$. A
Banach space is said to enjoy this property, if double-indexed random
sums admit a 'contraction principle':
\begin{align*}
   &  \; \widetilde{\EE} \,\EE\; \Bignorm{ \sum_{n\le N, k\le K}  \alpha_{n,k}  \widetilde{r_n} r_k  x_{n,k} }\\
  \le & \; 
  \max\bigl\{ |\alpha_{n,k}|: 1\le n \le N, \quad 1 \le k \le L\bigr\} \;
  \widetilde{\EE} \,\EE \; \Bignorm{ \sum_{n\le N, k\le K}    \widetilde{r_n} r_k  x_{n,k} }.
\end{align*}
As an example, this will be the case for all separable $L^p$-spaces,  Besov spaces or
Sobolev spaces.  Property $(\alpha)$ implies finite
cotype, but not necessarily non-trivial type of the Banach space
$X$ (think of $\ell_1$). We refer to \cite[Section~7.5]{FabulousFour:Band2}.

\section{Basic definitions of Ritt operators and their functional calculus}
\label{sec:Ritt-properties}

\smallskip

\noindent Ritt operators  are bounded operators on a Banach space
$X$, whose spectrum lies in the closed unit disc and that satisfy 
Ritt's condition \cite{Ritt1953}
\begin{equation}   \label{eq:Ritt}
  \exists K{>}0\quad \forall |\lambda|>1: \qquad  \norm{ (\lambda-1)R(\lambda, T) }\le K
\end{equation}
We say that $T$ is $\eR$-Ritt (respectively $\gamma$-Ritt), if the set
\begin{equation}   \label{eq:R-Ritt}
 \{  (\lambda-1)R(\lambda, T): \quad |\lambda|>1  \}
\end{equation}
is $\eR$-bounded (respectively $\gamma$-bounded).

\begin{figure}
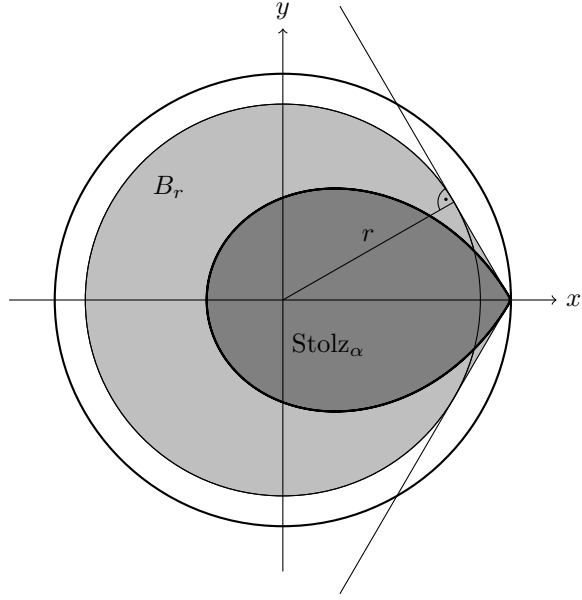

  \centering
  \StolzWinkel
  \caption{Two definitions of Stolz domains in the literature: as an
    example, for $\alpha=2$, the region $B_r$ in light grey, and the
    corresponding smaller Stolz region $\stolz{\alpha}$ in dark
    grey. Precise definitions of both regions will be given in section~\ref{sec:Ritt-properties} below.}
  \label{fig:2Stolz}
\end{figure}

\bigskip

When we talk about the $\Hinfty$-calculus of an operator, we need to
specify the domain of holomorphic functions. Functional calculi on
smaller domains are 'better', since a larger set of functions is
admissible.  By the classical Dunford-Riesz calculus, any bounded
operator operator admits a bounded functional calculus on a
sufficiently large domain, i.e. there exists a bounded algebra
homomorphism
\[
  \Phi: \left \{
    \begin{array}{lcl}
      \Hinfty(\calO) & \to & \BOUNDED(X) \\
      f & \mapsto & f(T) := \tfrac1{2\pi i} \displaystyle\int\limits_\Gamma f(z)\, R(z, T)\,dz
    \end{array}\right.
\]
for all complex domains $\calO$ that contain the spectrum of $T$.  The
point of Ritt operators is that the resolvent condition
\eqref{eq:Ritt} implies further spectral properties that allow to
narrow down the domain $\calO$ (initially a neighbourhood
$\overline{\DD}$) to smaller domains that are subsets of
$\DD \cup \{1\}$ where $\DD = \{ z \in \CC:\; |z|<1\}$.

\medskip

Indeed, evaluating condition \eqref{eq:Ritt} on the line
$\{ \Re(\lambda)=1 \}$ easily shows that the spectrum of $T$ is contained in
a symmetric sector around the real axis of angle $<\pihalbe$, centred
at $z=1$ and open to the left.  Moreover, the uniform boundedness of
the resolvent on
$\overline{\DD}^\complement \cap \{ Re(z) < 1 - \varepsilon\}$ shows
that the spectrum $\sigma(T)$ satisfies also
\[
  \sigma(T)
  \cap \{ \Re(z) < 1{-}\varepsilon\}  \quad \subseteq \quad B(0, R) \cap  \{ \Re(z) < 1{-}\varepsilon\}
\]
for some radius $R \in (0,1)$. Intersecting both spectral information,
one obtains for a suitable $r\in (0,1)$, necessarily
$\sigma(T) \subseteq B_r$, where $B_r$ is convex hull of the singleton
$z{=}1$ and the ball $B(0, r)$, see figure~\ref{fig:2Stolz}.  A more
refined analysis \cite[Proposition~4.2]{GomilkoTomilov} has shown that
the spectrum of a Ritt operator lies in a much smaller subset, that we
will call a {\em Stolz domain}: a set of the form
\begin{equation}
  \label{eq:stolz-def}
   \stolz{\omega} = \left\{ z \in \DD: \quad \frac{|1-z|}{1-|z|} < \omega \right\}
\end{equation}
for some $\omega>1$. We call $\omega$ the Stolz type of the
domain. It's opening angle is $2\arccos(\tfrac1\omega)$.

\medskip

\noindent We mention that a bounded operator $T$ on a Banach space is a {Ritt
  operator} if, and only the following two conditions hold simultaneously:
\begin{equation}
  \label{eq:powerbdd} \tag{\sc pb}
  \exists C_1 >0: \forall n\ge 0: \qquad  \norm{ T^n } \le C_1
\end{equation}
and
\begin{equation}
  \label{eq:discrete-derivative} \tag{\sc dd}
  \exists C_2 >0: \forall n\ge 0: \qquad  \norm{ k T^{k-1}(\Id - T)} \le C_2
\end{equation}
The first condition is {\em power-boundedness}, the second the {\em
  discrete derivative condition}. While the first one corresponds to
the boundedness of a $C_0$-semigroup, the latter corresponds to its
analyticity in a sector. None of the two conditions can be omitted
to characterise Ritt operators. Indeed, a simple 2-dimensional rotation
satisfies \eqref{eq:powerbdd} but not \eqref{eq:discrete-derivative},
whereas a counterexample to the opposite implication is presented in
\cite{KaltonMontgomerySmithOleszkiewiczTomilov}.
Similarly, $\eR$-Ritt means that the set
\[
  \{ T^n: \quad n\ge 1 \} \cup \{  k T^{k-1}(\Id - T): \quad k \ge 1\}
\]
is $\eR$-bounded in $\BOUNDED(X)$, see \cite[Lemma 5.2]{LeMerdy:Ritt2014}.

\bigskip

\begin{figure}
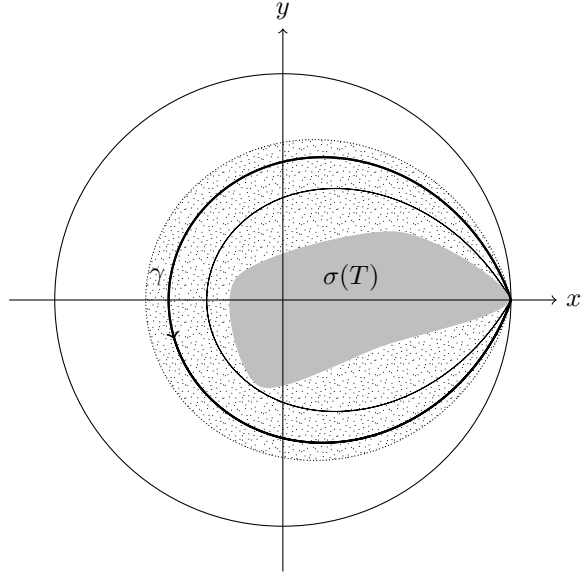

  \centering
    \RittKalkuel
  \caption{Spectrum and integration path for the elementary functional calculus of Ritt operators}
  \label{fig:Ritt-calc}
\end{figure}
\noindent Given a Ritt operator $T$ whose spectrum is contained in a
Stolz domain of type $\omega$ and given $\nu > \omega$ we let
\[
  \calE_{\omega, \nu} := \left\{ \varphi \in \Hinfty( \stolz{\nu}): \; \exists  \vartheta:\quad \omega < \vartheta < \nu\;\;\text{such that}\;\;  \tfrac{ \varphi(z) }{1-z} \in  L^1( \partial\stolz{\vartheta}) \right\}
\]
 For $f \in \calE_{\omega, \nu}$
we have the absolutely convergent integral
\begin{equation}  \label{eq:Ritt-elem-calc}
  \Phi(f) := f(T) := \tfrac1{2\pi i} \int_{\partial \stolz{\vartheta}} f(z) \, R(z, T)\,dz.
\end{equation}
that defines a bounded operator, see also figure~\ref{fig:Ritt-calc}.
We call the map $\Phi: \calE_{\omega, \nu} \to \BOUNDED(X)$ the \emph{
  elementary functional calculus} for $T$. We extend this by an
abstract approach outlayed in \cite{HaakHaase}: suppose that every
$f \in \Hinfty(\stolz{\nu})$ admits an elementary function
$e \in \calE_{\omega, \nu}$ such that $e(T)$ is injective and
$ef \in \calE_{\omega, \nu}$. We say that $e$ regularises $f$. We can
then define
\begin{equation}  \label{eq:abstr-fc-extension}
  \Phi(f) := f(T) := e(T)^{-1}  (ef)(T)
\end{equation}
and extend the functional calculus from $\calE_{\omega, \nu}$ to
$\Hinfty(\stolz{\nu})$: the operator $f(T) := \Phi(f)$ will be at
least a densely defined, closed operator that may, or may not, be
bounded: if so, we say that $T$ has a bounded
$\Hinfty(\stolz{\nu})$-functional calculus.

\smallskip

Since we will deal with upper, but also ``lower'' square function estimates,
i.e. estimates of the form
\[
  \EE \Bignorm{ \sum \gamma_k  k^{m_1-\einhalb} (\Id-T)^{m_1} T^{k-1} x }_{X}^2    \quad \ge \quad  c  \norm{x}
\]
for $T$ as well as its adjoint, we see that necessarily, $(\Id{-}T)$
is injective has dense range. This shows that such hypothesis are
natural in the given context, in particular in Theorem~\ref{thm:main}.
Here is another useful consequence: the Caley transform function
\[
  e(z) :=  \frac{1-z}{1+z}
\]
provides a nice elementary function that regularises all
$f \in \Hinfty$ simultaneously. Consequently, the construction given in 
\eqref{eq:abstr-fc-extension} above works well if $(\Id{-}T)$ is injective.

\medskip

Finally, usual functional calculus constructions require a 'weak'
control on the behaviour on approximating sequences : assume that the
uniformly bounded sequence $(f_n)$ converges pointwise to $f$ in
$\Hinfty$.  Then we need
\begin{equation}  \label{eq:bp-conv-requirement}
  \Phi(f_n)x \to \Phi(f)x
\end{equation}
weakly, for a dense set of points $x \in X$. Limiting ourselves to
points of the form $x=e(T)y$ for $y \in X$, and assume injectivity and
dense range of $(\Id{-}T)$ provides such a dense set, on which even
strong convergence follows easily from the dominated convergence
theorem.

\bigskip

Notice that the equivalent Ritt condition \eqref{eq:powerbdd} and
\eqref{eq:discrete-derivative} can be reformulated in terms bounded
operators acting form $X$ to $\ell_\infty(X)$ by
\[
  x \mapsto ( T^n x)_{n\ge 0} \qquad\text{and}\qquad x \mapsto (n
  (\Id-T)T^{n-1}x)_{n\ge 1}
\]
respectively. When we talk about square function characterisations of
the boundedness of the $\Hinfty$-calculus, this is done in a similar
way, but using different sequence norms. Given a separable Hilbert
space $H$ and a Banach space $X$, we say that a (bounded) operator
$\Phi: H \to X$ is $\gamma$-radonifying, if the (Gaussian) random
series
\[
  \sum_{n=1}^\infty  \gamma_n\, \Phi(h_n)
\]
converges in $L^2(\Omega; X)$, where $(h_n)$ is some orthonormal
basis. We refer to \cite{FabulousFour:Band2,Neerven:Survey} for more
details on the ideal of $\gamma$-radonifying operators.  In the
situation we consider in this note, we shall always let $H := \ell_2$.
Then $X$-valued sequences can then be identified with operators, by
defining $\Phi(e_n) := x_n$ on some orthonormal basis $(e_n)$ of
$\ell_2$. The $\gamma$-norm of a sequence $(x_n)$ is then defined as
the $\gamma$-norm of the corresponding operator $\Phi$:
\[
 \bignorm{ (x_n)_{n\ge 1} }_{\gamma(\ell_2; X)} :=  \left( \EE \Bignorm{\sum_n \gamma_n\, x_n }_X^2 \right)^\einhalb.
\]
In the case that $X$ is a Hilbert space, this is just the
$\ell_2(X)$-norm.

\bigskip

We say that our Ritt operator $T$ admits a square function estimate
$\Phi_m$, if  operator
\begin{equation}
  \label{eq:LeM-sqf}
  \Phi_m: \left\{
      \begin{array}{lcl}
    X   &\to&   \gamma(\ell_2; X) \\
    x    &\mapsto &  \bigl( k^{m-\einhalb} T^{k-1}(\Id-T)^m x \bigr)_k
      \end{array}\right.
  \end{equation}
  is bounded.   Along with square functions we consider ``dual'' square functions
  estimates.  These are formally bounded operators acting
\begin{equation}
    \label{eq:LeM-dual-sqf}
     \Phi_m^*: \left\{
      \begin{array}{lcl}
    X'   &\to&   \gamma(\ell_2; X)' \\
    x'    &\mapsto &  \bigl( k^{m-\einhalb} (T')^{k-1}(\Id-T')^m x' \bigr)_k
      \end{array}\right.    
  \end{equation}
  where $\gamma(\ell_2; X)'$ is the dual space of $\gamma(\ell_2; X)$,
  defined by trace duality, see for example
  \cite{FabulousFour:Band2,KaltonWeis:square-function-est,Neerven:Survey}:
  \begin{equation}    \label{eq:trace-duality-norm}
    \bignorm{ (x_n')_{n\ge 1} }_{\gamma(\ell_2; X)'} := \sup\left\{ \sum_n \dprod{x_n'}{x_n}: \quad  \bignorm{ (x_n) }_{\gamma(\ell_2; X)} \le 1 \right\}
  \end{equation}
  If $X$ has nontrivial type, then
  $\gamma(\ell_2; X)' = \gamma(\ell_2; X')$. This will be the case for
  reflexive $L^p$-spaces, reflexive Besov or Sobolev spaces, etc.

  \bigskip
  
  \noindent We will several times in the proofs use a handy
  observation, that appears as well in the proof of
  \cite[Theorem~3.3]{ArhancetLeMerdy:Ritt}.
\begin{lemma}\label{lem:stolz-geometric-series-lemma}
  Let $|u|<1$ and $n \ge 0$. 
  Then
\[
  \sum_{k=1}^\infty  \; \left(\!\!\!\begin{array}{c}k\\n\end{array}\!\!\! \right) \, (1-u)^{n+1} u^{k-1}  = 1
\]
\end{lemma}
\begin{proof}[Proof of Lemma~\ref{lem:stolz-geometric-series-lemma}]
  Since $|u|<1$, the geometric series $\sum_j u^j$ can be
  differentiated term-wise.  Therefore, we obtain
\[
      \sum_{k=1}^\infty  \; \left(\!\!\!\begin{array}{c}k\\n\end{array}\!\!\! \right) \, u^{k-1} 
\; =  \;  \tfrac{1}{n!} \left(\tfrac{\ud}{\ud{u}}\right)^n  \left(\sum_{j=0}^\infty   u^{j} \right) 
\; = \;  \frac{1}{(1-u)^{n+1}}. \qedhere
\]
\end{proof}

We are now ready to prove different results that, put together, show
in particular Theorem~\ref{thm:main}. They are split in three
parts. In section~\ref{sec:Hinfty-to-sqf} we give a new, and simple
method to infer square function estimates from the boundedness of the
functional calculus.  The approach also slightly improves
\cite[Thm. 7.3]{LeMerdy:Ritt2014} where only the case $m_1=m_2=1$ is
considered.  In the following section~\ref{sec:automatic-gamma-Ritt}
we show how square function estimates can be used to ``upgrade'' the
Ritt property to $\gamma$-Ritt or $\eR$-Ritt on Banach spaces with
property $(\alpha)$. Finally, in section~\ref{sec:form-sqf-to-hinfty}
we close the circle from square functions towards functional
calculus. This does not differ in essential points from Le~Merdy's
ideas developed in \cite{LeMerdy:Ritt2014}: it is a classical
``pushing through the square function'' technique that relies on the
fact that any operator has a bounded functional calculus with respect
to square-function norms.

 \section{Obtaining square functions estimates from the boundedness of
   the $\Hinfty$-calculus on Stolz domains}\label{sec:Hinfty-to-sqf}

 We will use an idea from \cite{HaakHaase} that in its most simplified
 version. To stay self-contained we re-do some of the arguments that
 will avoid us the more abstract notation from \cite{HaakHaase}.  Let
 us study the case of square function estimates $\Phi_m$ from
 \eqref{eq:LeM-sqf}. The key is to conceive them as the following
 functional calculus expression
  \[
    \Phi_m: \left\{
      \begin{array}{lcl}
        X &\to & \gamma(\ell_2; X) \\
        x & \mapsto & \left (  h \mapsto   \sprod{ F_m(z)}{ h } (T) x \right))
      \end{array}\right.
  \]
  where $F_m$ is a vector-valued function
  \begin{equation}
    \label{eq:sqf-as-functional-calulus}
        F_m: \left\{
          \begin{array}{lcl}
            \stolz{\omega} &\to &\ell_2, \\
            z &\mapsto & \left(k^{m-\einhalb} (1-z)^m z^{k-1}\right)_{k\ge 1}.
          \end{array}\right.
  \end{equation}
  Since we work with a Banach space $X$ of finite cotype, Gaussian and
  Rademacher random sums are equivalent. So let $(r_n)$ be a sequence
  of independent Rademacher variables and $(h_n)$ some sequence in $\ell_2$.
  Moreover, assume that $T$ has a bounded  $\Hinfty(\stolz{\nu})$-calculus on $X$, with $\nu > \omega$.
  Then   
\begin{equation}  \label{eq:Hinfty-to-sqf}
\begin{split}
 \EE  \Big\| \sum_{n=1}^N r_n \sprod{F_m}{ h_n } (T)  x \Big\| 
 & =   
  \EE \Big\|  \Bigl(\sum_{n=1}^N  r_n \sprod{F_m}{ h_n }\Bigr)  (T)  x \Big\| 
\\ & \le 
     C_{\Hinfty}(T) \norm{x} \; \EE   \sup_{z \in \stolz{\nu}} \Bigl| \sum_{n=1}^N  r_n \sprod{F_m(z)}{ h_n } \Bigr| \\
  & \le 
     C_{\Hinfty}(T) \norm{x} \;   \sup_{z \in \stolz{\nu}}  \Bignorm{ \left(\sprod{ F_m(z)}{ h_n }\right)_k }_{\ell_1}.
\end{split}
\end{equation}
If we let $(h_n)$ an orthonormal basis we arrive precisely at the
definition $\gamma$-norms (actually, uniform boundedness of these
partial sums, and convergence of the full random series are equivalent
in Banach spaces not containing $c_0$ and a fortiori on spaces of
finite cotype).  Due to the ``ideal property'' of $\gamma$-radonifying
operators (cf. e.g. \cite[Theorem 6.2]{Neerven:Survey}), we have even
more freedom, for example to select $(h_n)$ as a Riesz basis (i.e. an
isomorphic image of an orthonormal basis).

\medskip

If we consider the vector-valued function
\eqref{eq:sqf-as-functional-calulus}, the estimate given in
\eqref{eq:Hinfty-to-sqf} tells us that we obtain the desired square
function estimate $\Phi_m$ from the boundedness of the functional
calculus, if we can pick a Riesz basis $(h_n)$ in a way that the
scalar products $(\sprod{ F_m(z)}{ h_n })_{n\ge 1}$ form an absolutely
summable sequence, whose $\ell_1$-norm is uniformly bounded on Stolz
domains. While the square summability of the sequence of these scalar
products is not sensitive to the choice of the Riesz-basis $(h_n)$,
absolute sums are. Let us illustrate this in the case $m=1$. We write
then $F=F_1$.

\medskip

Taking scalar products of $F(z)$ against a canonical orthonormal basis
$(e_n)$ of $\ell_2$, we obtain
\[
  \sprod{ F(z)}{ e_n } = \sqrt{n}  (1-z) z^{n-1}
\]
and the $\ell_1$-norm of this sequence equals
\[
  \sum_n |\sprod{ F(z)}{ e_n }| = |1-z| \; \text{Li}_{{-}\tfrac12}(|z|)
\]
where $\text{Li}_{s}(z)$ are the poly-logarithmic functions. By
Wirtinger's theorem \cite{Wirtinger1905} (see also, for example
\cite{NavasRuizVarona}) it has the following asymptotics when
$|z| <1, z \to 1$:
\begin{equation}  \label{eq:Wirtinger}
  \text{Li}_{{-}\nicefrac{1}{2}}(z) \quad \sim \quad \Gamma(\nicefrac{3}{2}) (1-z)^{-3/2}
\end{equation}
as a consequence, the $\ell_1$-sums are not uniformly bounded on Stolz
domains. This simply means that $h_n = e_n$ is a bad choice of a basis.

\medskip

Let us now devise another basis of $\ell_2$ that is better behaved.
We will write $a := e^{2\pi i / 3}$, one non-trivial third root of
unity, and we let $b := i$ for a forth root.    With these, we build the
two $5\times 5$ matrices
\[
  A = \left(    
  \begin{array}{lllll}
    1 &  a\phantom{{}^2} &   a^2            & 0       & 0 \\
    1 &  a^2             &   a^4            & 0       & 0 \\
    0 & 1                &   b\phantom{{}^2}& b^2  &  b^4 \\
    0 & 1                &   b^2            & b^4  &  b^6 \\
    0 & 1                &   b^3            & b^6  &  b^9 \\
  \end{array}
\right)
\qquad
\text{and}
\qquad
D_k = \text{diag}(  \sqrt{\tfrac{5k+1}{5k+j}} )_{j=1..5}
\]
Since $A$ is invertible and $D_k$ has clearly uniformly bounded
condition numbers, also $A_k = A D_k$ has clearly uniformly bounded
condition numbers. By identifying $\ell_2$ with $\bigoplus_2 \CC^5$ we
define the operator $A = \bigoplus A_k$ on $\ell_2$ and see that it is
an isomorphism of $\ell_2$. Consequently, the lines of the infinite
block diagonal matrix
\[
  \left(
    \begin{array}{cccc}
      A_0 & 0   & 0   & 0 \ldots \\
      0   & A_1 & 0   & 0 \ldots \\
      0   & 0   & A_2 & 0 \ldots \\
   \vdots & \vdots & 0 & \ddots
    \end{array}\right)
  \]
  form a Riesz basis $\{ b_1, b_2, b_3, \ldots \}$ of $\ell_2$. This
  basis has some nice features that the canonical basis of $\ell_2$
  does not have. Indeed, we observe that
  \[
    \begin{array}{lll}
      \sprod{ F(z)}{ b_1 } & = \sqrt{1} (1+az+a^2 z^2) (1-z)   &= \sqrt{1} (1-z) \phantom{z^2}\, \tfrac{1-z^3}{1-az} \\
      \sprod{ F(z)}{ b_2 } & = \sqrt{2} (1+a^2z+a^4 z^2) (1-z) z &= \sqrt{2} (1-z) z\phantom{{}^2} \, \tfrac{1-z^3}{1-a^2z} \\
      \sprod{ F(z)}{ b_3 } & = \sqrt{3} (1+bz+b^2 z^2 + b^3 z^3 ) (1-z) z^2 &= \sqrt{3} (1-z) z^2 \, \tfrac{1-z^4}{1-bz} \\
      \sprod{ F(z)}{ b_4 } & = \sqrt{4} (1+b^2z+b^4 z^2 + b^6 z^3) (1-z) z^3 &= \sqrt{4} (1-z) z^3 \, \tfrac{1-z^4}{1-b^2z} \\
      \sprod{ F(z)}{ b_5 } & = \sqrt{5} (1+b^3z+b^6 z^2 + b^9 z^3) (1-z) z^4 &= \sqrt{5} (1-z) z^4 \, \tfrac{1-z^4}{1-b^3z} \\
      \sprod{ F(z)}{ b_6 } & = \sqrt{6} (1+az+a^2 z^2) (1-z)^5   &= \sqrt{6} (1-z) z^5 \, \tfrac{1-z^3}{1-az} \\
      \sprod{ F(z)}{ b_7 } & = \sqrt{7} (1+a^2z+a^4 z^2) (1-z) z^6 &= \sqrt{7} (1-z) z^6 \, \tfrac{1-z^3}{1-a^2z} \\
    \end{array}
  \]
  and so forth: compared to the scalar product of $F(z)$ against a
  canonical basis vector $e_n$ we gain an additional regularising
  factor $(1-z)$ in the numerator of each fraction, and ``pay'' the
  gain with a pole at some non-trivial root of unity by the
  denominator --- but these poles lie all outside of Stolz domains:
  they do not harm the uniform boundedness of $\ell_1$-norms for
  $z \in \stolz{\omega}$.

  In virtue of the Stolz domain condition \eqref{eq:stolz-def}, the
  growth order of $(1-|z|)^{-3/2}$ that we saw in \eqref{eq:Wirtinger}
  will now be (more than) compensated by the factor $|1-z|^2$: as a
  consequence, our approach outlayed in \eqref{eq:Hinfty-to-sqf} does
  work when we use the Riesz-basis $(b_n)$ instead of the canonical
  orthonormal basis $(e_n)$. We obtain the following result.

\begin{proposition}\label{prop:ritt-calc-implies-LeM-sqf}
  Let $X$ be a Banach space of finite cotype, and $T$ a Ritt operator
  of type $\omega$ on $X$. If $T$ has a bounded $\Hinfty(\stolz{\nu})$-calculus
  for some $\nu>\omega$, then for any $m \ge 1$
  \begin{align*}
    & \Phi_m: \left\{
      \begin{array}{lcl}
    X   &\to&   \gamma(\ell_2; X) \\
    x    &\mapsto &  \bigl( k^{m-\einhalb} T^{k-1}(\Id-T)^m x \bigr)_k
      \end{array}\right.\\
    \qquad\text{and}\qquad &
    \Phi_m^*: \left\{
      \begin{array}{lcl}
    X'   &\to&   \gamma(\ell_2; X)' \\
    x'    &\mapsto &  \bigl( k^{m-\einhalb} (T')^{k-1}(\Id-T')^m x' \bigr)_k
      \end{array}\right.
  \end{align*}
     define  bounded (dual) square functions. 
   \end{proposition}
\begin{proof}
  The proof of the case $m>1$ for square function estimates is a
  straightforward modification of the case $m=1$ explained above, and
  therefore omitted. In virtue of the preceding discussion, the
  square function estimate for $\Phi_m$ is clear, so we only have to
  explain the dual square functions. Recall the trace duality from
  \eqref{eq:trace-duality-norm}. It inherits the ideal property from
  $\gamma$-norms, i.e. we may pass from an orthonormal basis to a
  Riesz basis.  For a sequence $(x_k) \in \gamma(\ell_2; X)$ we
  consider
\begin{align*}
          & \; \sum_k  \dprod{ \sprod{F_m(z)}{b_k}(T)' x'}{x_k} \\
      \lesssim  & \; \EE \dprod{x'}{\sum_k r_k  \sprod{F_m(z)}{b_k}(T) \left( \sum_j r_j  x_j \right)} \\
     \le & \;  C_{\Hinfty}(T) \norm{x'} \sup_{z \in \stolz{\nu}}  \Bignorm{ \left(\sprod{ F_m(z)}{ b_k } \right)_k}_{\ell_1}  \; \EE \Bignorm{\sum_j r_j  x_j }\\
     \le & \;  \sqrt{\tfrac{\pi}2} \, C_{\Hinfty}(T) \; \norm{x'} \sup_{z \in \stolz{\nu}}  \Bignorm{ \left(\sprod{ F_m(z)}{ b_k }\right)_k }_{\ell_1}  \; \EE \Bignorm{\sum_j \gamma_j  x_j }.
\end{align*}
The last inequality does not use finite cotype, but a simple domination of Rademacher sums by Gaussian sums, see \cite[Prop. 12.11]{DiestelJarchowTonge}.  Taking now the supremum
over all sequences $(x_n)$ with $\norm{ (x_n) }_{\gamma} \le 1$ finishes to proof.
\end{proof}

\section{'Ritt property'  via square function estimates}\label{sec:automatic-gamma-Ritt}

In this section we extend \cite[Theorem~5.3]{LeMerdy:Ritt2014} to
Banach spaces with property $(\alpha)$. Often, property $(\alpha)$ is
used to improve properties, like 'upgrading' Ritt property to the
$\eR$-Ritt property. Notice that we go further here: we show that merely
(dual) square function estimates for a bounded operator suffice to
obtain the $\eR$-Ritt property, under the mild geometrical property
$(\alpha)$.  Recall that $\eR$-boundedness and $\gamma$-boundedness are
equivalent notions on Banach spaces with property $(\alpha)$.

\begin{theorem}\label{thm:producing-R-Ritt}
  Let $T$ be a bounded operator on a Banach space having property
  $(\alpha)$ such that $\Id+T$ is invertible.  Assume that $T$ admits
  square functions $\Phi_{m_1}$ for $m_1\ge 1$ and dual square
  functions $\Phi_{m_2}^*$, where $m_2\ge 1$.

  Then $T$ is $\eR$-Ritt (and $\gamma$-Ritt).
\end{theorem}

Recall
that property $(\alpha)$ implies finite cotype, so that Rademacher--
and Gaussian averages are equivalent: $\gamma$-Ritt can be replaced by
$\eR$-Ritt in this section.

\begin{proof}
\textbf{Step 1:} {\em $T$ is an $\gamma$-power-bounded operator}.

\noindent The main idea is to observe that $T^n$ acts as a shift on
the square functions of the type $\Phi_{m_1}, \Phi_{m_2}$.  This gives
the idea to ``push'' the operator $T^n$ through the square-function to
obtain $\eR$-boundedness.  To this end we need lower square function
estimates. Up to some minor modification, these come out of
Lemma~\ref{lem:stolz-geometric-series-lemma} that we read as an
approximate identify: Clearly, for any $m \ge 1$,
\[
  k^{m-\einhalb} \approx \sqrt{k} (k+1)(k+2)\cdots(k+m-1)
\]
in the sense of a double inequality with uniform constants for all
$k>0$.  As a consequence of the contraction
principle~\eqref{eq:contraction-principle}, and the fact that
$(\Id{+}T)$ and $(\Id{+}T')$ are isomorphisms acting on $X$ and $X'$
respectively, the supposed (dual) square functions for
$\Phi_{m_1}, \Phi_{m_2}^*$ imply (dual) square functions for
\begin{align*}
  & \widetilde{\Phi}_{m_1} :\; 
  \left\{
      \begin{array}{lcl}
    X   &\to&   \gamma(\ell_2; X) \\
    x    &\mapsto &  \bigl( \sqrt{k} (k+1)\cdots(k+m_1-1) T^{k-1}(\Id-T)^{m_1} (\Id+T)^{-m_1} x \bigr)_k
      \end{array}\right.\\
\text{and}\quad & 
  \widetilde{\Phi}_{m_2} :\; 
  \left\{
      \begin{array}{lcl}
    X   &\to&   \gamma(\ell_2; X) \\
    x    &\mapsto &  \bigl( \sqrt{k} (k+1)\cdots(k+m_2-1) T^{k-1}(\Id-T)^{m_2} (\Id+T)^{-m_2} x \bigr)_k
      \end{array}\right.
\end{align*}
The corresponding functional calculus expressions are given by
\begin{align*}
  & \widetilde{\varphi}_{m_1}(z) = \sqrt{k} (k+1)\cdots(k+m_1-1) z^{k-1}(1-z)^{m_1} (1+z)^{-m_1} \\
  \text{and}\quad &
  \widetilde{\varphi}_{m_2}(z) = \sqrt{k} (k+1)\cdots(k+m_2-1) z^{k-1}(1-z)^{m_2} (1+z)^{-m_2} 
\end{align*}
Now using Lemma~\ref{lem:stolz-geometric-series-lemma} again, we have
\[
  \sprod{ \widetilde{\varphi}_{m_1}(z)}{ \widetilde{\varphi}_{m_2}(z) }_{\ell_2}
  =  \sum_{k=1}^\infty  \; k(k+1)\cdots(k+m-1) (1-z)^{m} z^{2k-2} (1+z)^{-m} = m! 
\]
By Theorem~\ref{thm:identity-and-upper-give-lower}, this implies {\em
  lower} square function estimates for $\widetilde{\Phi}_{m_1}$ and
hence for ${\Phi}_{m_1}$. We are ready to prove
$\gamma$-boundedness. Let $N\ge 1$ and $x_1, \ldots, x_N \in X$ be
given. Then property $(\alpha)$ allows to use the contraction
principle for double-indexed random sums in the following estimate%
:
 \begin{align*} 
   & \; \EE \Bignorm{ \sum_{n=1}^N \gamma_n T^n x_n }\\
\text{(lower sqf)}  \quad \lesssim & \;   \EE  \;\widetilde \EE \Bignorm{ \sum_{n=1}^N \sum_{k=1}^\infty \gamma_n  \widetilde \gamma_k k^{m_1-\einhalb }(I-T)^{m_1}  T^{n+k-1}  x_n }\\
\text{(property $(\alpha$))}  \quad  \lesssim & \;   \EE  \;\widetilde \EE \Bignorm{ \sum_{n=1}^N \sum_{k=1}^\infty \gamma_n \widetilde\gamma_k  (k+n)^{m_1-\einhalb }(I-T)^{m_1}  T^{n+k-1}  x_n }\\
(j=k{+}n)  \quad   = & \;    \EE  \;\widetilde \EE \Bignorm{ \sum_{n=1}^N   \sum_{j\ge n} \gamma_n \widetilde \gamma_j j^{m_1-\einhalb }(I-T)^{m_1}  T^{j-1}  x_n }\\
\text{(property $(\alpha$))} \quad   \le & \;    \EE  \;\widetilde \EE \Bignorm{ \sum_{n=1}^N \sum_{j\ge 1} \gamma_n   \widetilde \gamma_j j^{m_1-\einhalb }(I-T)^{m_1}  T^{j-1}  x_n }\\
\text{(upper sqf)}  \quad   \le & \;  \EE  \Bignorm{  \sum_{n} \gamma_n x_n }.
 \end{align*}
 We obtain that $\{ T^n: n \ge 0 \}$ is $\gamma$-bounded. The discrete derivatives are obtained in two further steps.

\smallskip
 
 \textbf{Step 2}: {\em Obtaining a special $\gamma$-bounded set } \\
\noindent We start with the following formula that relies on the geometric
series for $z^2$:
\begin{align*}
    & \; (2k-2)^{m-1} (1-z)^{m-1} z^{2k-2}  \\
  = & \; 2(1+z)  \sum_{j=k}^\infty   (k-1)^{m-1}  \times (1-z)^{m} z^{2j-2}\\
  = & \; 2(1+z)  \sum_{j=k}^\infty   \left(\tfrac{k-1}{j}\right)^{m-1}  \times  ( j^{m_1-\einhalb}(1-z)^{m_1} z^{k-1})   \times (j^{m_2 -\einhalb }(1-z)^{m_2} z^{k-1})
\end{align*}
The front factor $2(\Id+T)$ is an isomorphism and therefore unimportant to us.
The idea is to conceive this sum as an 'integral representation' with the
multipliers
\[
  m_k:  j \mapsto \eins_{j\ge k} \left(\tfrac{k-1}{j}\right)^{m-1},
\]
and to appeal to Theorem~\ref{thm:int-repres}. Clearly,
$|m_k(j)| \le 1$.  By the assumed (dual) square-function estimates and
the fact that $X$ has property $(\alpha)$, the integral representation
theorem~\ref{thm:int-repres} yields that
\[
    \GB{ k^{m-1} (\Id-T)^{m-1} T^{2k-2}: k \ge 1 } \lesssim  \norm{ {\Phi_{m_1}} }  \norm{ \Phi_{m_2}^* }
\]
A similar calculation can be done for odd powers of $T$, by pulling
out the factor $z$ in the representation formula and reducing it to an even power. Summarising,
\[
    \GB{ k^{m-1} (\Id-T)^{m-1} T^{k-1}: k \ge 1 } \lesssim  \norm{ {\Phi_{m_1}} } \norm{ \Phi_{m_2}^* }.
\]

\smallskip

\noindent \textbf{Step 3:} {\em Discrete derivatives form a  $\gamma$-bounded set}\\
\noindent As a consequence of lemma
\ref{lem:stolz-geometric-series-lemma} for the choice $n=m-2$, we have
\begin{align*}
  n(1-T)T^n  =
  \left\{
  \begin{array}{ll}
    \; \displaystyle \sum_{k=n}^\infty  \tfrac{n}{k^2} \times k^2(1-T)^2 T^k  & \text{ if } m=3 \\
    {} & {} \\
    \; \displaystyle \sum_{k=n}^\infty  \tfrac{n(k-n+1)}{k^3} \times k^3(1-T)^3 T^k & \text{ if } m=4 \\
    {} & {} \\
   \; \displaystyle  \sum_{k=n}^\infty \tfrac{n\, c_{m-3}(k-n+1) }{(m-4)! \; k^{m-1}} \times k^{m-1} (\Id-T)^{m-1} T^k & \text{ if } m\ge 5.
  \end{array}
\right.
\end{align*}
The scalar multipliers satisfy in each case, uniformly over all $n \ge 1$,
\[
  \sum_{k=n}^\infty \tfrac{n}{k^2} \lesssim 1,
  \qquad
   \sum_{k=n}^\infty  \tfrac{n(k-n+1)}{k^3} \lesssim 1
   \quad\text{and}\quad
   \sum_{k=n}^\infty n \tfrac{c_{m-3}(k-n+1) }{(m-4)! \; k^{m-1}} \;\lesssim 1
 \]
By the convex-hull property \eqref{eq:convex-hull}, and the previous step,
\[
  \GB{ n (1-T) T^n : n\ge 1 } \quad \lesssim \quad  \norm{ \Phi_\gamma( \Phi_{m_1} ) }  \norm{ \Phi_{\gamma'}( \Phi_{m_2}^* ) }
\]
follows. The discrete derivatives of $(T^n)$ form therefore a
$\gamma$-bounded set.  Together with our
 findings in Step 1, we proved that
 \[
   \{ T^n, n (\Id-T)T^{n-1}: \qquad n\ge 1 \}
 \]
 is $\eR$-bounded, which by \cite[Lemma 5.2]{LeMerdy:Ritt2014} is equivalent to $T$
 being  $\eR$-Ritt or $\gamma$-Ritt.
\end{proof}

We stress the importance of (dual) square function estimates for these
arguments: of course a bounded operator needs not to be Ritt. But even
if we assumed the Ritt property, instead of the (dual) square function
estimates, we could not conclude: indeed, as stated in
\cite[Corollary~1.3.2]{Arnold:PhD} and the subsequent remarks, any
Banach space $X$ that admits an unconditional basis while not being
isomorphic to a Hilbert space allows to construct Ritt operators which
are not $\eR$-Ritt. This shows that our hypotheses of square function
estimates cannot be weakened easily.

\section{Obtaining a bounded $\Hinfty$-calculus on Stolz domains via
  square functions estimates }\label{sec:form-sqf-to-hinfty}

In this section we give a similar proof to Le~Merdy's result
\cite[Theorem~7.3]{LeMerdy:Ritt2014}.  This allows to strengthen
step~1 of the proof of Theorem~\ref{thm:producing-R-Ritt} and push
$f(T)$ through the square function. We re-formulate the proof,
although being very close for the sake of completeness, but also since
the domain of the underlying functional calculus (i.e. $\stolz{\omega}$ instead of
$B_r$) is different.

\begin{theorem}\label{thm:R-Ritt-and-sqf-yield-Hinfty}
  Let $X$ be a Banach space and an $\eR$-Ritt operator of type
  $\omega$ on $X$ that
  admits (dual) square function estimates $\Phi_{m_1}, \Phi_{m_2}^*$.
  Then $T$ has a bounded $\Hinfty(\stolz{\theta})$-calculus for all
  $\theta>\omega$.
\end{theorem}

Observe that if $X$ has Pisier's property $(\alpha)$, the $\eR$-Ritt
condition is not needed thanks to Theorem~\ref{thm:producing-R-Ritt}.

\begin{proof}
  Let $f \in \Hinfty(\stolz{\nu})$ and where $\nu > \omega$ and fix some $\omega < \vartheta < \nu$.
  We start applying Lemma~\ref{lem:stolz-geometric-series-lemma} for
  $n=m+1$ and $u=z^3$ with $|z|<1$: multiplying with $f(z)$, we obtain
  the representation formula
  \begin{equation}    \label{eq:repres-formula}
  f(z) =  \sum_{k=1}^\infty m_k(z) \times \varphi_{m_1}(k,z) \times \varphi_{m_2}(k,z). 
  \end{equation}
where
\[
  m_k(z) = \tfrac{1}{(m+1)!} f(z) (1+z+z^2)^{m+1}  \left(\prod_{j=1}^{m} \tfrac{k+j}{k} \right)  \times k (1-z) z^{k-1}
\]
are elementary functions. If we had assumed $X$ to have property $(\alpha)$, then
\begin{equation}  \label{eq:claim}
  \GB{ m_k(T): k \ge 1 } \lesssim   C_\vartheta \norm{f}_\infty.
\end{equation}
follows right away from the second part of
Theorem~\ref{thm:int-repres}. But the result is true on any Banach
spaces, as we will show below.  Let us first put into light its usefulness:
\begin{align*}
    & \quad \left|\dprod{ f(T)x}{x'}\right|\\
  = & \quad \left|\dprod{ \sum_{k=1}^\infty \sprod{\varphi_{m_2}(z) m_k(z)  \varphi_{m_1}(z)}{e_k}(T)x }{x'} \right|\\
  = & \quad \left|\dprod{ \EE    \left(\sum_{n=1}^\infty \gamma_n \, \sprod{\varphi_{m_2}(z)}{e_n} (T) \right) \left( \sum_{k=1}^\infty \gamma_k  \, m_k(T)   \sprod{\varphi_{m_1}(z)}{e_k} (T)x \right)}{x'} \right|\\
  \le & \quad \left( \EE \Bignorm{ \sum_{k=1}^\infty \gamma_k  \, m_k(T)   \sprod{\varphi_{m_1}(z)}{e_k} (T)x }_X^2\right)^\einhalb
        \left( \EE \Bignorm{\sum_{n=1}^\infty \gamma_n \sprod{\varphi_{m_2}(z)}{e_n} (T)'x'}_{X'}^2 \right)^\einhalb \\
  \overset{\text{\eqref{eq:claim}}}{\le} \!\!
   &  \quad C_\vartheta  \norm{f}_\infty \;
        \norm{ {\Phi_{m_1}} x}_{\gamma}  \norm{ \Phi_{m_2}^* x' }_{\gamma'}\\
 \le &  \quad C_\vartheta  \norm{f}_\infty \; \norm{ {\Phi_{m_1}} }  \norm{ \Phi_{m_2}^* } \;          \norm{x}_X \norm{x'}_{X'}.
\end{align*}
and taking the supremum over all $\norm{x'} \le 1$ and $\norm{x}\le 1$
the result follows.

\medskip

\noindent It remains to prove the claim. First, the scalar factors
\[
  \tfrac{1}{(m+1)!} \left(\prod_{j=1}^{m} \tfrac{k+j}{k} \right)
\]
in $m_k$ can be removed by the contraction principle
\eqref{eq:contraction-principle}. Moreover, $\Id+T+T^2$ is an
isomorphism by the spectral mapping theorem. This means that we may
focus on the operators $\{ f(T) \, k(\Id - T)T^{k-1}: k \ge 1\}$.  To
treat the $\eR$-boundedness of this set, a functional calculus argument
is used: the boundary of $\Gamma := \partial \stolz{\vartheta}$ is
parameterised by
 \begin{equation}  \label{eq:stolz-parametrisation}
  \gamma(t) := 1 - r(t) e^{it} \qquad r(t) = \tfrac{2\vartheta}{\vartheta^2-1} (\vartheta \cos(t) - 1),
 \end{equation}
 where $t \in [{-}C_\vartheta, C_\vartheta]$ with
 $C_\vartheta := \arccos(\tfrac1\vartheta)$.  Since
 \[
      f(T) \, k(\Id - T)T^{k-1} x  = \tfrac{1}{2\pi\ui} \int_{\Gamma}  k z^{k-1} f(z) (\Id-T) R(\lambda, T) x \ud{z}
 \]
 it is sufficient to verify that $z \mapsto k z^{k-1} f(z)$ is absolutely integrable on $\Gamma$, satisfying
 $\norm{  k z^{k-1} f(z) }_{L^1(\Gamma)} \le \norm{f}_\infty$  to conclude 
 by the convex-hull lemma that
 \[
    \RB{  f(T) \, k(\Id - T)T^{k-1} : k \ge 1 } \quad\lesssim\quad  \norm{ f }_{\Hinfty(\stolz{\nu})}. 
 \]
 Similar arguments appear in Le~Merdy's paper \cite{LeMerdy:Ritt2014}
 who refers essentially to Vitse \cite{Vitse:Ritt}.  Her estimate is
 done for the Stolz type domain $B_\vartheta$ instead of
 $\stolz{\vartheta}$, so we provide full details for our setting
 here, even if there are some similarities.  First notice that
   \begin{align*}
   |\gamma(t)| = & \frac{1+\vartheta^2 - 2\vartheta \cos(t)}{\vartheta^2-1} = 1 - \tfrac{r(t)}{\vartheta}
   \qquad\text{and}\\
   |\gamma'(t)| = & \frac{2\vartheta}{\vartheta^2-1} \sqrt{1+\vartheta^2 - 2\vartheta \cos(t)}
                = \tfrac{ 2\theta }{ \sqrt{\vartheta^2-1}}\, |\gamma(t)|^\einhalb 
   \end{align*}
 Therefore, using that $c \le |\gamma(t) | \le 1$, we find that
 \begin{align*}
   \int_\Gamma  k |z|^{k-1} |f(z)|\, |\ud{z}|
   \lesssim_\vartheta   & \; \norm{f}_{\Hinfty(\stolz{\nu})} \;  \int_{-C_\vartheta}^{C_\vartheta} k |\gamma(t)|^{k-\einhalb}\, dt\\
   \eqsim_\vartheta  & \; \norm{f}_{\Hinfty(\stolz{\nu})} \;  \int_{-C_\vartheta}^{C_\vartheta} k |\gamma(t)|^{k-1}\, dt.
 \end{align*}
 We split this integral in two parts: the first part where
 $t \in (-\eps, \eps)$ is far away from $z{=}1$. The second part
 $|t|\ge \eps$ contains the point $z{=}1$. We fix some $\eps>0$, taken
 sufficiently small to satisfy $\vartheta \cos(\eps) > 1$.  In the
 first region, we have $|\gamma(t)|\le q < 1$ and so
 $\sup_k k q^{k} \le \sup_{t>0} t q^{-t} = |\ln(q)|e^{-1}$ shows
  \[
    \int_{|t|< \eps } |f(\gamma(t))| \;  k |\gamma(t)|^{k-1} \,\ud{t} \le  \norm{f}_\infty  \; |\Gamma|\, |\ln(q)|e^{-1}.
  \]
 When $|t|\ge \eps$, we  use a change of variables: since 
 $r'(t) = \tfrac{2\vartheta^2}{\vartheta^2-1} \sin(t)$,
 \begin{align*}
   \int_{\eps< |t| < C_\vartheta} |f(\gamma(t))| \; k |\gamma(t)|^{k-1} \,\ud{t}
   \lesssim_\vartheta  & \; \norm{f}_\infty \;  \int_{\eps< |t| < C_\vartheta} k \left( 1-\tfrac{r(t)}{\vartheta} \right)^{k-1} \,\ud{t}\\
   \lesssim_{\vartheta}  & \; \norm{f}_\infty \; \tfrac{1}{\sin(\eps)}  \int_{\eps< |t| < C_\vartheta}  k \sin(t) \left( 1-\tfrac{r(t)}{\vartheta} \right)^{k-1}     \,dt \\ 
   = & \; \norm{f}_\infty \; \tfrac{1}{\sin(\eps)} \int_0^{\tfrac{2}{\vartheta^2-1}(\alpha\cos(\eps)-1)} k \left( 1-s \right)^{k-1} \,\ud{s} \\
   \le & \; \norm{f}_\infty \; \tfrac{1}{\sin(\eps)} \left[ -(1-s)^k \right]_{s=0}^{s=\tfrac{2}{\vartheta^2-1}(\alpha\cos(\eps)-1)} \\
   \le & \; \tfrac{1}{\sin(\eps)} \norm{f}_\infty.
 \end{align*}
 and the proof is complete: we refrain from optimising over all
 possible $\eps>0$ in both inequalities.
\end{proof}

\subsection*{Open question: }
\noindent For a given Ritt operator, the link between its $\Hinfty$
functional calculus and its associated square function estimates
resembles the analogous link for sectorial (or strip type) operators
$A$. In fact, the proofs we present are built up using the same
abstract principles.

\noindent For both, Ritt and sectorial operators on Banach spaces of
finite cotype, the functional calculus (if it is bounded) can be
extended to so-called ``bounded square functional'' (or ``quadratic'')
$\Hinfty$-calculus, verifying
\begin{equation}  \label{eq:sqf-calculus}
 \EE \Bignorm{ \sum_n \gamma_n  f_n(T)x  } \quad \lesssim \quad \sup_z \left( \sum_n |f_n(z)|^2 \right)^{\einhalb} \norm{x}.
\end{equation}
In the case of sectorial (or strip type) operators, the square
functional calculus can be deduced directly from
Theorem~\ref{thm:int-repres} via  integral representations
that make direct use of ``standard'' square functions.
In the case of Ritt operators however, this does not seem to
work. The integral representation analogous to the sectorial
situation would be
\begin{equation}
  \label{eq:expect}
  f(z) = \sum_k  a_k \; \varphi_{m_1}(k,z) \varphi_{m_2}(k,z)
\end{equation}
where the coefficients $a_k$ do not depend on the variable $z$ (unlike
the formula \eqref{eq:repres-formula} above). Despite this lack, the
validity of \eqref{eq:sqf-calculus} was shown in Le~Merdy's paper
\cite{LeMerdy:Ritt2014}.  Instead of using the assumed (dual) square
functions of the functions $\varphi_m(z)$, he rather bootstraps the
square functional calculus from the (scalar) $\Hinfty$-calculus by
means of the Francks-McIntosh decomposition that produce a kind of
``brute force square functions'' for a collection of functions which
are not explicitly known. Our integral representation
theorem~\ref{thm:int-repres} can then be applied to the
Francks-McIntosh decomposition.

\noindent From this observation two question emerge: first, we would like
to identify  the (non-separable) subspaces
\[
  \calS_{m_1, m_2} := \left\{  \sum_k  a_k \; \varphi_{m_1}(k,z) \varphi_{m_2}(k,z): \qquad (a_k)\in \ell_\infty \right\}
\]
and their sum in $ \Hinfty(\stolz{\omega})$.  Second, we expect that
other, explicit square functions for Ritt operators are to be found
that will, allow amongst others, general representation formulas. This
question might be linked to the conformal map of Stolz regions to the
unit disc, that seem not to be explicitly known, either.

\appendix

\section{Tools from abstract functional calculus theory}\label{sec:tools}

We cite in this section to abstract results borrowed from
Haak-Haase \cite{HaakHaase}.  We base our results on these results to emphasise
that the $\Hinfty$-calculus theory of Ritt operators as well as that
of sectorial or strip type operators can be explained with a unified
abstract approach to holomorphic functional calculus. For both
theorems we require that $(\Hinfty(\calO), \Phi)$ be a functional
calculus on the Banach space $X$ satisfying the abstract construction
principles outlayed in \eqref{eq:abstr-fc-extension} and
\eqref{eq:bp-conv-requirement}: in particular they do apply to Ritt
operators $T$, for which $(\Id{-}T)$ is injective and has rense range.

\begin{theorem} \label{thm:int-repres}
  Let $K$ be a Hilbert space and $H := L^{2}(\Omega)$ for
  some measure space $(\Omega, \mu)$. Suppose further that
  $f,\: g \in \Hinfty(\calO;H)$ such that the  square
  function associated with $g$,
  \[
    \Phi_{\gamma}(g): X\to \gamma(H;X),
  \]
  as well as the dual square function associated with $f$,
  \[
    \Phi_{\gamma'}(f): X'\to \gamma'(H;X'),
  \]
  are bounded.  Consider, for $m \in L^{\infty}(\Omega;K')$, the
  function $u\in \Hinfty(\calO;K')$ defined by
\begin{equation}\label{eq:int-repres}
  u(z) := \int_\Omega m(t) \cdot f(t,z)\, g(t,z)\, \mu(\ud{t})\,\,  \in\,\,  K' \qquad (z\in \calO).
\end{equation}
If $K$ is finite-dimensional or if $X$ has finite cotype and
 $x\in \DOMAIN(\Phi_\gamma(g))$  then  $x\in  \DOMAIN(\Phi_\gamma(u))$
 and
\begin{equation}\label{eq:int-repres-conclusion}
 \norm{\Phi_\gamma(u)x}_\gamma
\le C\, \norm{m}_{L^{\infty}(\Omega;K')} \, 
\norm{\Phi_{\gamma'}(f)}
\norm{\Phi_\gamma(g)x}_\gamma,
\end{equation}
where $C$ depends only on  $\dim(K)$ or the cotype (constant) of $X$, respectively.

\medskip

\noindent The following strenghening holds true: if $X$ has
additionally property $(\alpha)$, then the set
\[
  \left \{ \Phi(u_m) \suchthat m \in L^{\infty}(\Omega,\mu),\,
\norm{m}_\infty \le 1 \right\}
\]
is $\gamma$-bounded, with bound
\[
  \GB{\Phi(u_m) \suchthat \norm{m}_\infty \le 1}
\le C^{+}C^{-} \norm{\Phi_\gamma(g)}_\gamma 
\norm{\Phi_{\gamma'}(f)}_{\gamma'},
\]
where $C^{+}C^{-}$ is the condition number
of the isomorphism
\[
  \gamma(\ell_2 \oplus \ell_2; X) \;\simeq\; \gamma(\ell_2; \gamma(\ell_2;X)).
\]
\end{theorem}

The following result form \cite{HaakHaase} gives an abstract tool to
infer lower square function estimates for a function $g$ from upper
ones of a function $f$ under an ``identity condition'' on the pointwise
scalar products $\sprod{f(z)}{g(z)}_H$.

\begin{theorem}\label{thm:identity-and-upper-give-lower}
  Suppose that $f \in \Hinfty(\calO;H)$ and $g\in \Hinfty(\calO;H')$
  are such that $\Phi_\gamma(g)$ and $\Phi_{\gamma'}(f)$ are bounded
  operators.  If $\sprod{f(z)}{g(z)}_H = 1$ for all $z \in \calO$,
  one has the norm equivalence
  \[
       \norm{x}_X \simeq \norm{\Phi_\gamma(g)x}_{\gamma}  \quad \text{for all $x\in X$}.
  \] 
\end{theorem}

\def\SUBMITTED{Submitted}
\def\TOAPPEAR{To appear in }
\def\PREPARATION{In preparation }

\def\cprime{$'$}
\providecommand{\bysame}{\leavevmode\hbox to3em{\hrulefill}\thinspace}

\end{document}